\documentclass[a4paper,10pt]{scrartcl}
\usepackage[latin1]{inputenc}
\usepackage{amsmath,amsthm}
\usepackage{amsfonts}
\usepackage{amssymb}
\usepackage{a4wide}
\usepackage{todonotes}
\DeclareMathOperator*{\argmin}{argmin}
\newcommand{\norm}[2]{\left\Vert #1\right\Vert_{#2}}
\newcommand{\X}{\mathcal X}
\newcommand{\Y}{\mathcal Y}
\newcommand{\V}{\mathcal V}
\newcommand{\fdag}{f^\dagger}
\newcommand{\gobs}{g^{\mathrm{obs}}}
\newcommand{\gdag}{g^\dagger}
\newcommand{\R}{\mathbb R}
\newcommand{\N}{\mathbb N}
\newcommand{\falhat}{\hat f_\alpha}
\newcommand{\T}[1]{\mathcal T\left(#1;\gdag\right)}
\renewcommand{\S}[1]{\mathcal S\left(#1;\gobs\right)}
\newcommand{\Prob}[1]{\mathbb P \left[#1\right]}
\newcommand{\E}[1]{\mathbb E \left[#1\right]}
\newcommand{\supp}{\mathrm{supp}}

\def\fad{f_\alpha^\delta}
\def\fdag{f^\dag}

\newtheoremstyle{slanted}
{}
{}
{\slshape}
{}
{\bfseries}
{.}
{ }
{}
\newtheoremstyle{roman}{}{}{\rmfamily}{}{\bfseries}{.}{ }{}

\theoremstyle{plain}
\newtheorem{thm}{Theorem}
\newtheorem{cor}{Corollary}

\newtheorem{pro}{Proposition}
\newtheorem{ass}{Assumption}
\theoremstyle{slanted}
\newtheorem{rem}{Remark}
\theoremstyle{roman}
\newtheorem{ex}{Example}

\title{Convergence Analysis of (Statistical) Inverse Problems under Conditional Stability Estimates}
\author{\textsc{Frank Werner}\thanks{University of Goettingen, Institute for Mathematical Stochastics, Goldschmidtstraße 7, 37077 G\"ottingen, Germany, e-mail: f.werner@math.uni-goettingen.de.} \,\, and \textsc{Bernd Hofmann}\thanks{Faculty of Mathematics, Chemnitz University of Technology, 09107 Chemnitz, Germany, \newline e-mail: hofmannb@mathematik.tu-chemnitz.de.}}
\date{\today}

\begin{document}

\maketitle

\begin{quote}
\textbf{Abstract:} Conditional stability estimates require additional regularization for obtaining stable approximate solutions if the validity area of such estimates is not completely known. In this context, we consider ill-posed nonlinear inverse problems in Hilbert scales satisfying conditional stability estimates characterized by general concave index functions. For that case, we exploit Tikhonov regularization and provide convergence and convergence rates of regularized solutions for both deterministic and stochastic noise. We further discuss a priori and a posteriori parameter choice rules and illustrate the validity of our assumptions in different model and real world situations.
\end{quote}

\begin{quote}
\noindent
{\small \textbf{Keywords:}
Statistical inverse problems,
conditional stability,
Hilbert scales,
Tikhonov regularization,
Lepskij principle,
convergence rates}
\end{quote}

\begin{quote}
\noindent
{\small \textbf{AMS-classification (2010):}
47J06, 65J20, 47A52
}
\end{quote}

\section{Introduction} \label{sec:intro}

In this paper, we investigate the operator equation
\begin{equation} \label{eq:opeq}
F(f)=g,
\end{equation}
which acts as model of an inverse problem with a (possibly nonlinear) \textit{forward operator}
\[
F : D(F) \subseteq \X \to \Y,
\]
with domain $D(F)$ mapping between the infinite dimensional separable real Hilbert spaces $\X$ with norm $\|\cdot\|$ and
$\Y$ with norm $\|\cdot\|_{\Y}$. Let denote by $\fdag \in D(F)$ the uniquely determined solution of
(\ref{eq:opeq}) for the exact right-hand side $g=F(\fdag) \in \Y$. Moreover let, as is typical for inverse problems,
(\ref{eq:opeq}) be \textit{locally ill-posed} at $\fdag$, which means that for closed balls $B_r(\fdag):=\{f \in D(F)~\big|~\|f-\fdag\|\le r \}$ around $\fdag$
with arbitrarily small radii $r>0$ there exist sequences $\{f_n\}_{n=1}^\infty \subset B_r(\fdag)$ such that $\liminf \limits_{n \to \infty}\|f_n-\fdag\| >0$,
but $\lim\limits_{n \to \infty}\|F(f_n)-F(\fdag)\|_{\Y} =0$ (cf., e.g., \cite[Def.~3]{HofPla18}). Consequently, in order to find \textit{stable approximate solutions} to equation
(\ref{eq:opeq}) based on observed noisy data $\gobs$ of $g$, some kind of stabilization is required. Our focus here is on \textit{variational regularization in a Hilbert scale} under \textit{conditional stability estimates}.

For considering the Hilbert scale  we introduce a densely defined (unbounded and closed) linear self-adjoint operator $L\colon D(L)
\subset \X \to \X$, which is strictly positive such that we have for some $m>0$
\begin{equation} \label{eq:m}
\|Lx\|\geq m \|x\|\qquad \mbox{for all} \qquad x\in D(L).
\end{equation}
The operator~$L$ satisfying \eqref{eq:m} generates a Hilbert scale $\{\X_\nu\}_{\nu \in
  \mathbb{R}}$ with $\X_0:=\X$, $\X_\nu=D(L^\nu)$,  and with
corresponding  norms $\|x\|_\nu:=\|L^\nu x\|_\X$.
It is well-known that for a triple of indices $-a < t \leq s$ the \textit{interpolation inequality}
\begin{equation}\label{eq:interpolationX}
\|f\|_t \leq \|f\|_{-a}^{\frac{s-t}{s+a}} \|f\|_s^{\frac{t+a}{s+a}}
\end{equation}
holds for all $f \in \X_s$.

\smallskip

In the following, we will consider a mixed data model, which allows to treat both deterministic and stochastic error contributions. Therefore recall the notion of a Hilbert space process $Z$ on $\Y$, which is a bounded linear mapping $Z : \Y \to \mathbf L^2 \left(\Omega, \mathcal A, \mathbb P\right)$ with a probability space $\left(\Omega, \mathcal A, \mathbb P\right)$. Note that, by definition, $\Prob{Z \in \Y} = 0$, and it is common to write $\left\langle Z, g\right\rangle:= Z \left(g\right)$ for $g \in \Y$. A Hilbert space process $Z$ is called centered, if $\E{\left\langle Z,g\right\rangle} = 0$ for all $g \in \Y$, and it is called white, if $\text{Cov}\left[\left\langle Z, g_1\right\rangle, \left\langle Z, g_2\right\rangle\right] = \left\langle g_1, g_2\right\rangle$ for all $g_1, g_2 \in \Y$. A Hilbert space process $Z$ is called Gaussian, if $\left(\left\langle Z, g_1\right\rangle, ..., \left\langle Z, g_n\right\rangle\right)$ follows a multivariate Gaussian distribution for any choice of $g_1, ..., g_n \in \Y$ and $n \in \N$. With this notion in mind, we consider the data model
\begin{equation}\label{eq:stochastic_noise_model}
\gobs = \gdag + \sigma Z + \delta \xi
\end{equation}
with a centered Gaussian white noise $Z$ on $\Y$, some (deterministic) element $\xi \in \Y$ with $\norm{\xi}{\Y} \leq 1$, and parameters $\sigma, \delta > 0$. Model \eqref{eq:stochastic_noise_model} covers both deterministic and stochastic error contributions, parameterized by $\delta$ and $\sigma$ respectively, see \cite{BHMR07} for examples. Note that if $\sigma = 0$, then $\gobs \in \Y$, the measurements $\gobs$ at hand are purely deterministic and satisfy the classical bound
\begin{equation}\label{eq:noiselevel}
\|\gobs-g\|_{\Y} \leq \delta
\end{equation}
with the \textit{noise level} $\delta>0$. In this case we concretize the situation by assigning $\gobs=g^\delta$. If $\sigma > 0$, then $\Prob{\gobs \in \Y} = \Prob{Z \in \Y} = 0$ and hence \eqref{eq:stochastic_noise_model} has to be understood in a weak sense, this is for each $g \in \Y$ we observe
\begin{equation}\label{eq:weak_noise}
\left\langle \gobs ,g\right\rangle = \left\langle \gdag,g\right\rangle + \delta \left\langle g, \xi\right\rangle + \sigma \left\langle Z, g\right\rangle,
\end{equation}
where, by definition, $\left\langle Z, g\right\rangle$ is a random variable with distribution $\mathcal N \left(0,\norm{g}{\Y}^2\right)$, and for two elements $g_1, g_2 \in \Y$ the dependency structure is encoded in $\text{Cov}\left[\left\langle Z, g_1\right\rangle,\left\langle Z, g_2\right \rangle\right] = \E{\left\langle Z, g_1\right\rangle \left\langle Z, g_2\right \rangle} = \left\langle g_1, g_2\right\rangle$.

\smallskip

Initially, we pose two assumptions which are valid throughout the paper. The first assumption refers to properties of $F$, $D(F)$ and $\fdag$. Moreover, it defines occurring indices $a$ and $s,u$ in the Hilbert scale
under consideration.

\begin{ass} \label{ass:basic1}
\begin{itemize} \item[]
\item[(a)] The domain $D(F)$ of $F$ is a convex and closed subset of $\X$.
\item[(b)] The operator $F: D(F) \subseteq \X\to \Y$ is weak-to-weak sequentially continuous, i.e. $x_n \rightharpoonup x_0$ in $\X$ with $x_n \in D(F),\;n \in \N,$ and $\bar x \in D(F)$ implies $F(x_n) \rightharpoonup F(\bar x)$ in $Y$.
\item[(c)] There exists $u > 0$ such that $\fdag \in \X_u$, where $\fdag \in D(F)$ is the unique solution to (\ref{eq:opeq}) for given right-hand side $g$.
\item[(d)] There are further indices $a,s \in \R$ such that $a \ge 0,$ $0 \le s < u \le 2s+a$, and $-a<s$.
\end{itemize}
\end{ass}

\smallskip

In the following we will need closed balls and their intersections with the domain of definition $D(F)$ of $F$, this is
$$B^\nu_\mu(\bar f) := \left\{f \in \X_\nu~\big|~ \|f-\bar f\|_\nu \leq \mu\right\}, \qquad D^\nu_\mu\left(\bar f\right) := B^\nu_\mu(\bar f)\cap D(F)$$
in $\X_\nu\;(\nu \in \mathbb{R})$ with center $\bar f \in \X_\nu$  and radius $\mu\;(0<\mu \le \infty)$, where we write for simplicity $B_\mu(\bar f)$ and $D_\mu(\bar f)$ instead of $B^0_\mu(\bar f)$ and $D^0_\mu(\bar f)$, respectively. Now we are in position to introduce the second assumption in form of a conditional stability estimate.
\begin{ass} \label{ass:basic2}
There are a concave index function\footnote{We call a function $\varphi\colon [0,\infty) \to [0,\infty)$ index function if it is continuous, strictly increasing and satisfies the boundary condition $\varphi(0) = 0$.} $\varphi$, values $\theta \ge 0$, $\;\rho>0$ and $R>0$ as well as a subset $Q$ of $D_\rho^\theta\left(\fdag\right) \subset D(F)$ such that the conditional stability estimate
\begin{equation}\label{eq:stability}
\|f - \fdag\|_{-a} \leq R \, \varphi\left(\|F(f)-F(\fdag)\|_{\Y}\right)
\end{equation}
holds for all $f \in Q$, where the multiplier $R$ may depend on $a,\;\varphi$ and $Q$.
\end{ass}

\smallskip

There are two main sources for verifying conditional stability estimates of the form \eqref{eq:stability}:
\begin{itemize}
\item[(A)] Local structural conditions for the nonlinearity of $F$,
\item[(B)] Global inequalities of the forward operator $F$.
\end{itemize}
In general, the local nonlinearity conditions in (A) require G\^ateaux or Fr\'echet derivatives $F^\prime(f)$ in the solution point $f=\fdag$ or in small intersected balls $f \in D_r(\fdag)$ around $\fdag$, whereas the global inequalities
in (B) do not need derivatives of $F$ at all. Such global inequalities typically occur in parameter identification problems for partial differential equations, sometimes in connection with Carleman estimates.
In the Appendix we present for motivation and illustration three examples for relevant sets $Q$. Feasible elements $f$, for which the conditional stability estimate \eqref{eq:stability} is valid, belong
in all three examples to the intersection of $D(F)$ with one or two closed balls of the type $B^\nu_\mu(\bar f)$. The Examples~\ref{ex:Q1} and \ref{ex:Q2} in the Appendix refer to local conditions in the sense of (A), whereas
Examples~\ref{ex:Q3} and \ref{ex:stab} are based on global inequalities in the sense of (B).

\bigskip

Under the stated assumptions we search for approximate solutions $\falhat$ to $\fdag$, which are \textit{regularized solutions} as minimizers
\begin{equation}\label{eq:tik_gen}
\falhat \in \argmin_{f\in D(F)} \left[ \S{F\left(f\right)} + \alpha \norm{f}{s}^2\right]
\end{equation}
of the Tikhonov functional with $s$-norm square penalty $\norm{f}{s}^2$ and a data fidelity term $\S{\cdot}$. If $\sigma = 0$ in \eqref{eq:stochastic_noise_model}, i.e. if we have deterministic data $\gobs = g^\delta \in\Y$, we will consider the most common choice $\S{g} = \frac12 \norm{g - g^\delta}{\Y}^2$, i.e. \eqref{eq:tik_gen} specializes to
\begin{equation}\label{eq:tik}
\falhat \in \argmin\limits_{f\in D(F)} \left[ \frac12 \norm{F\left(f\right) - g^\delta}{\Y}^2 + \alpha \norm{f}{s}^2 \right]
\end{equation}

If $\sigma > 0$ in \eqref{eq:stochastic_noise_model}, then one has $\gobs \notin \Y$ with probability $1$ as discussed above, and hence $\norm{g-\gobs}{\Y}= +\infty$ a.s. for any $g \in\Y$. However, the functional $\T{g} = \frac12 \norm{g - \gdag}{\Y}^2$ can still be interpreted as an ideal data fidelity term, which is unavailable (as $\gdag$ is unknown). In view of \eqref{eq:weak_noise} it seems natural to use $\S{g} := \frac12\norm{g}{\Y}^2 - \left\langle g,\gobs\right\rangle$ as data fidelity term in that case, which ensures well-definedness and formally differs from $\T{\cdot}$ only by the additive constant $\frac12 \norm{\gobs}{\Y}^2$ (which is however $+\infty$ in the stochastic case). Hence, for stochastic noise we consider
\begin{equation}\label{eq:tik_stat}
\falhat \in \argmin_{f\in D(F)} \left[ \frac12\norm{F\left(f\right)}{\Y}^2 - \left\langle F\left(f\right),\gobs\right\rangle + \alpha \norm{f}{s}^2\right].
\end{equation}

Note that the penalty $f\mapsto\norm{f}{s}^2$ is a non-negative, convex and sequentially lower semi-continuous functional. By definition of the Hilbert scale, for all $s\ge 0$, this functional is \textit{stabilizing} in the sense that all its sublevel sets are weakly sequently compact in $\X$. Under Assumption~\ref{ass:basic1}, \textit{existence and stability} of approximate solutions $\falhat$ in the sense of \cite[Section~4.1.1]{ScKaHoKa12} are then evident, since Assumptions~3.11 and 3.22 in \cite{ScKaHoKa12} are satisfied (in case of stochastic noise, this is a.s. the case). Moreover, note that we always have for the minimizer of the Tikhonov functional $\falhat \in \X_s$, which means that there is a radius $\overline \rho>0$ such that $\falhat$ and $\fdag$ both belong to $D^s_{\overline \rho}(0)$. In order to obtain \textit{convergence} of the regularized solutions to $\fdag$, the interplay of the noise magnitude and the choice of the regularization parameter $\alpha>0$ must be appropriate.

To prove even \textit{convergence rates} in variational regularization, \textit{smoothness conditions} have to be imposed on $\fdag$. It will be shown that the conditional stability estimate \eqref{eq:stability} from  Assumption~\ref{ass:basic2} allows us to verify error estimates and convergence rates for the constructed approximate solutions and that the property
$\fdag \in \X_u \cap D(F)$ is sufficient to serve as such a smoothness condition if the index $u$ matches the set $Q$ from \eqref{eq:stability}.
In this context, however, we should emphasize that the stability estimate \eqref{eq:stability} is not powerful enough to yield alone stable approximate solutions to \eqref{eq:opeq} since $Q$ is in general not or not completely known. Therefore, the additional use of Tikhonov-type regularization is needed in order to force the approximate solutions into the set $Q$ of admissible elements for \eqref{eq:stability} for sufficiently small noise.

In the context of smoothness conditions we also mention commonalities between conditional stability estimates \eqref{eq:stability} and variational source conditions, which have become a major tool to derive convergence rates during the last decade. In case of the Hilbert scale regularization \eqref{eq:tik} and adapted to \eqref{eq:stability}, variational source conditions attain the form
\begin{equation} \label{eq:vsc}
\|f-\fdag\|_{-a} \le \|f\|_s^2-\|\fdag\|_s^2+ R\,\varphi(\|F(f)-F(\fdag)\|_{\Y})\qquad \mbox{for all}\quad f \in M,
\end{equation}
valid for some set of admissible elements $M$. Variational source conditions of the form \eqref{eq:vsc} with $\varphi\left(t\right) = \sqrt{t}$ have been introduced in \cite{HoKaPoSc07} and appeared recently for example in \cite{BurFleHof13,Fle18,Flemmingbuch18,Gras10,HofMat12,HohWei15,Scherzerbuch09,ScKaHoKa12}. Similar to conditional stability estimates, variational source conditions express in an implicit way both nonlinearity conditions and solution smoothness of the underlying nonlinear inverse problem.
	
There is a certain connection between conditional stability estimates \eqref{eq:stability} and variational source conditions, which depends very much on the set $M$. Since the difference  $\|f\|_s^2-\|\fdag\|_s^2$ may attain positive and negative values for varying $f \in M$, there is no immediate connection. However, if $M$ is such that the roles of $f$ and $\fdag$ in \eqref{eq:vsc} can be interchanged, then each variational source condition immediately implies a conditional stability estimate as examined in \cite{HohWei15}. If \eqref{eq:vsc} is validated based on spectral source conditions and nonlinearity estimates, this will in general not be the case. For another approach to variational source conditions with general convex penalty functionals in the Tikhonov regularization of linear problems we refer to \cite{HKM19,Kind16}.
	
More recently, there have been approaches (see e.g. \cite{HohWei15,HohWei17,HohWei17b,WeiSprHoh18}) to verify variational source conditions directly for specific problem instances without relying on nonlinearity assumptions or spectral source conditions or on both. In this case, it can happen that the set $M$ allows to interchange $f$ and $\fdag$ in \eqref{eq:vsc}, and hence also a conditional stability estimate follows, see e.g. Example \ref{ex:stab} in the appendix.

\bigskip

The remainder of the paper is organized as follows: The focus of Section~\ref{sec:Det} is on convergence and convergence rate assertions for deterministic inverse problems. As main result of Section~\ref{sec:Det}, in Theorem~\ref{thm:deterministc} and its Corollary \ref{cor:rates} convergence rates for general concave index functions $\varphi$ in the conditional stability estimate  \eqref{eq:stability} are formulated and proven. This section closes a gap in the theory by extending the results recently published in \cite{EggerHof18} from the H\"older case to the case of general concave index functions. Section~\ref{sec:Stat} is the statistical counterpart to Section~\ref{sec:Det} with Theorem~\ref{thm:stochastic} and Corollary \ref{cor:stoch} as main result concerning convergence rates. In the Appendix we finally discuss a series of motivating examples.

\section{Deterministic Inverse Problems} \label{sec:Det}

In this section we consider a \textit{deterministic noise model}, this is \eqref{eq:stochastic_noise_model} with $\sigma = 0$. Recall that this implies $\gobs \in \Y$, $\|\gobs-g\|_{\Y} \leq \delta$ and we write $\gobs=g^\delta$ and $\falhat=f_\alpha^\delta$. Based on Assumption~\ref{ass:basic1} the following proposition on \textit{convergence} is an immediate consequence of \cite[Theorem~4.3 and Corollary~4.6]{ScKaHoKa12}. In this context, we also take into account the usual properties of Hilbert scales, moreover the Kadec-Klee property of Hilbert spaces and the fact that $\fdag$ is assumed to be the unique solution to \eqref{eq:opeq}
and sufficiently smooth.

\begin{pro} \label{pro:convergence}
Let $\alpha=\alpha(\delta)$ (a priori choice) or $\alpha=\alpha(\delta,g^\delta)$ (a posteriori choice) be choices of the regularization parameter $\alpha>0$ satisfying the limit conditions
\begin{equation} \label{eq:conv}
\alpha \to 0 \qquad \mbox{and} \qquad \frac{\delta^2}{\alpha} \to 0 \qquad \mbox{as} \qquad \delta \to 0,
\end{equation}
then we have under Assumption~\ref{ass:basic1} and for $\delta_n \to 0$ as $n \to \infty$, $\alpha_n=\alpha(\delta_n)$ or $\alpha_n=\alpha(\delta_n,g^{\delta_n})$, and $f_n=f_{\alpha_n}^{\delta_n}$
\begin{equation} \label{eq:conv1}
\lim \limits_{n \to \infty} \|F(f_n)-g\|_{\Y}=0,
\end{equation}
\begin{equation} \label{eq:conv2}
\lim \limits_{n \to \infty} \|f_n\|_s = \|\fdag\|_s,
\end{equation}
and
\begin{equation} \label{eq:conv3}
\lim \limits_{n \to \infty} \|f_n-\fdag\|_\nu=0\qquad \mbox{for all} \qquad 0 \le \nu \le s.
\end{equation}
\end{pro}

Based on conditional stability estimates required by Assumption~\ref{ass:basic2}, however, we can even prove \textit{convergence rates} for the regularized solutions. We remark that the set $Q$ of admissible elements with
associated radii and Hilbert scale indices and the index function $\varphi$ in this assumption need not be known. On the other hand, as long as the
choice of the regularization parameter $\alpha>0$ obeys the condition (\ref{eq:conv}), we have by formula (\ref{eq:conv3}) from Proposition~\ref{pro:convergence} that for fixed $\nu \in [0,s]$ and arbitrarily small radii $\rho>0$ there is some $\overline \delta>0$  such that $f_\alpha^\delta \in D^\nu_\rho(\fdag)$ whenever $0<\delta \le \overline \delta$.

In the following we will employ some convex analysis. The \textit{Fenchel conjugate} of a function $h : \mathbb R \to \bar{\mathbb{R}}$ is defined by $h^* \left(y\right):= \sup_{x \in \mathbb R} \left[ xy - h\left(x\right) \right]$. For an index function $h$ (defined on $\left[0,\infty\right)$) the Fenchel conjugate can be defined accordingly by extending $h$ to all of $\mathbb R$ by setting $h \left(-x\right) := \infty$ for $x > 0$, which leads to
\[
h^* \left(y\right):= \sup_{x \geq 0} \left[ xy - h\left(x\right) \right].
\]
Note that $h^*$ is always convex as a supremum over affine linear functions, and that for convex $h$ it holds $\left(h^*\right)^* = h$. For such $h$ we furthermore denote by $\partial h\left(x\right)$ the \textit{subdifferential} of $h$, i.e.
\[
\partial h \left(x\right) = \left\{y \in \R ~\big|~ h\left(z\right) \geq h\left(x\right) + y\left(z-x\right) \text{ for all } z \in \R\right\}.
\]	
The \textit{Fenchel-Young inequality} states that
\begin{equation}\label{eq:fenchel_young}
ab \leq h\left(a\right) + h^* \left(b\right)
\end{equation}
for all $a,b \in \R$ with equality if and only if $a \in \partial h^* \left(b\right)$, which for convex $h$ is in turn equivalent to $b \in \partial h\left(a\right)$. For more details on convex analysis we refer to \cite{r97}.

Now we are ready to formulate our first main theorem, which yields an error decomposition.
\begin{thm} \label{thm:deterministc}
Let the Assumptions~\ref{ass:basic1} and \ref{ass:basic2} hold and let the regularization parameter $\alpha>0$ be chosen a priori or a posteriori such that for sufficiently small noise levels $0<\delta \le \bar \delta$ the regularized solutions $\fad$ belong to the set $Q$ of admissible elements of the conditional stability estimate \eqref{eq:stability}. Then we have for such $\delta$ with the function
\[
\psi_{u,s,a}(t):=\left(\varphi(\sqrt{t})\right)^{\frac{2(u-s)}{a+u}},\qquad t>0,
\]
depending on the concave index function $\varphi$ and on the indices $a,s,u$ the error estimate
\begin{equation} \label{eq:errorest1}
\|\fad-\fdag\|_s^2 \le \frac{\delta^2}{\alpha}+C(-\psi_{u,s,a})^*\left(-\frac{1}{8C\alpha} \right)
\end{equation}
with a constant $C = C \left(R, \norm{\fdag}{u}, u,s,a\right)$.
\end{thm}
\begin{proof}
By assumption we have $\fad, \fdag \in Q$ for all $0<\delta \le \bar \delta$. Hence using \eqref{eq:interpolationX} and \eqref{eq:stability} we can compute
\begin{align*}
\norm{\fdag}{s}^2 - \norm{\fad}{s}^2 + \norm{\fad-\fdag}{s}^2 &\leq 2\norm{\fdag}{u} \norm{\fad - \fdag}{2s-u}\\
& \leq 2 \norm{\fdag}{u} \norm{\fad - \fdag}{-a}^{\frac{u-s}{s+a}} \norm{\fad - \fdag}{s}^{\frac{2s+a-u}{s+a}}\\
& \leq 2R \norm{\fdag}{u}  \norm{\fad - \fdag}{s}^{\frac{2s+a-u}{s+a}} \varphi\left(\norm{F\left(\fad\right) - \gdag}{\Y}\right)^{\frac{u-s}{s+a}}.
\end{align*}
Next we apply Young's inequality (this is just \eqref{eq:fenchel_young} with $h(a) = a^p/p$ and $h^*(b) = b^q/q$) in the form
\begin{equation}\label{eq:young}
ab = \left(\varepsilon a\right) \left(\frac{b}{\varepsilon}\right) \leq \frac{\varepsilon^p}{p} a^p + \frac{1}{q \varepsilon^q} b^q
\end{equation}
with $\varepsilon = \left(p/(4R\norm{\fdag}{u})\right)^{1/p}$, $p = \frac{2s+2a}{2s+a-u}$ and $q = \frac{2s+2a}{a+u}$. This yields
\begin{align*}
\norm{\fdag}{s}^2 - \norm{\fad}{s}^2 + \norm{\fad-\fdag}{s}^2 \leq \frac{1}{2} \norm{\fad - \fdag}{s}^2 + C \varphi\left(\norm{F\left(\fad\right) - \gdag}{\Y}\right)^{\frac{2u-2s}{a+u}}
\end{align*}
with a constant $C = C \left(R, \norm{\fdag}{u}, u,s,a\right)$. Thus we have
\begin{equation}\label{eq:aux1}
\norm{\fdag}{s}^2 - \norm{\fad}{s}^2 + \frac12 \norm{\fad-\fdag}{s}^2 \leq C \varphi\left(\norm{F\left(\fad\right) - \gdag}{\Y}\right)^{\frac{2u-2s}{a+u}}.
\end{equation}
It follows from the minimizing property of $\fad$ in \eqref{eq:tik}, that
\[
\frac12 \norm{F\left(\fad\right) - \gobs}{\Y}^2 + \alpha \norm{\fad}{s}^2  \leq \frac12 \norm{F\left(\fdag\right) - \gobs}{\Y}^2 + \alpha \norm{\fdag}{s}^2  \leq \frac{\delta^2}{2} + \alpha \norm{\fdag}{s}^2
\]
where we used \eqref{eq:noiselevel}. Due to the triangle inequality and $\left(a+b\right)^2 \leq 2a^2 + 2b^2$ it holds
\[
\norm{F\left(\fad\right) - \gdag}{\Y}^2 \leq 2\norm{F\left(\fad\right) - \gobs}{\Y}^2  + 2\delta^2
\]
which hence implies
\[
\frac14 \norm{F\left(\fad\right) - \gdag}{\Y}^2 - \frac{\delta^2}{2} + \alpha \norm{\fad}{s}^2  \leq \frac{\delta^2}{2} + \alpha \norm{\fdag}{s}^2
\]
Some rearranging yields
\begin{equation}\label{eq:aux2}
\frac18\norm{F\left(\fad\right) - \gdag}{\Y}^2 \leq \delta^2 + \alpha \left(\norm{\fdag}{s}^2 -  \norm{\fad}{s}^2\right) - \frac18 \norm{F\left(\fad\right) - \gdag}{\Y}^2.
\end{equation}
Combining \eqref{eq:aux1} and \eqref{eq:aux2} gives
	\begin{align}
	\frac18\norm{F\left(\fad\right) - \gdag}{\Y}^2  + \frac{\alpha}{2} \norm{\fad-\fdag}{s}^2 & \leq \delta^2 + \alpha \left(\norm{\fdag}{s}^2 -  \norm{\fad}{s}^2 + \frac12\norm{\fad-\fdag}{s}^2 \right)- \frac18 \norm{F\left(\fad\right) - \gdag}{\Y}^2\nonumber\\
	& \leq \delta^2 + C\alpha \varphi\left(\norm{F\left(\fad\right) - \gdag}{\Y}\right)^{\frac{2u-2s}{a+u}} - \frac18 \norm{F\left(\fad\right) - \gdag}{\Y}^2 \nonumber \\
	& \leq \delta^2  + C\alpha\sup_{\tau \geq 0} \left[\varphi\left(\tau\right)^{\frac{2u-2s}{a+u}} - \frac{1}{8C\alpha} \tau^2\right] \nonumber\\
	& \leq \delta^2 +C\alpha \left(-\psi_{u,s,a}\right)^*\left(-\frac{1}{8C\alpha}\right). \label{eq:error_decomposition_det}
	\end{align}
The claim follows by from dividing by $\alpha$.
\end{proof}

\begin{rem}
Note that the assumption in Theorem \ref{thm:deterministc} that for sufficiently small noise levels $0<\delta \le \bar \delta$ the regularized solutions $\fad$ belong to the set $Q$ of admissible elements of the conditional stability estimate \eqref{eq:stability} is satisfied if the choice of the regularization parameter satisfies the condition \eqref{eq:conv} and if the set $Q$ is the intersection of a finite number of closed intersected balls $D_\rho^\nu(\fdag)$ with $0 \le \nu \le s$.
\end{rem}

Before we conclude with convergence rates under a priori and a posteriori parameter choice rules, let us collect some facts about the approximation error in \eqref{eq:errorest1}:
\begin{rem}\label{rem:app_err}
Let
\[
\varphi_{\mathrm{app}} \left(\alpha\right) := \left(-\psi_{u,s,a}\right)^*\left(-\frac{1}{\alpha}\right) =  \sup_{\tau \geq 0}\left[\psi_{u,s,a}\left(\tau\right) - \frac{\tau}{\alpha}\right], \qquad \alpha > 0.
\]
\begin{enumerate}
\item[(a)] As $\psi_{u,s,a} \left(0\right) = 0$ we obtain $\varphi_{\mathrm{app}} \left(\alpha\right) \geq 0$ for all $\alpha > 0$.
\item[(b)] As $\psi_{u,s,a}$ and $\alpha \mapsto -\frac{\tau}{\alpha}$ for fixed $\tau>0$ are monotonically increasing, we also find that $\varphi_{\mathrm{app}}$ is monotonically increasing.
\item[(c)] The concavity of $\varphi$ together with $\varphi\left(0\right) = 0$ implies that
\begin{equation}\label{eq:concavity_psiusa}
\psi_{u,s,a}\left(C\tau\right) = \varphi\left(\sqrt{C}\sqrt{\tau}\right)^{\frac{2\left(u-s\right)}{a+u}} \leq \left(\sqrt{C} \varphi\left(\sqrt{\tau}\right)\right)^{\frac{2\left(u-s\right)}{a+u}} = C^{\frac{u-s}{a+u}} \psi_{u,s,a}\left(\tau\right)
\end{equation}
for any $C>1, \tau>0$. Thus it holds
\begin{align*}
\varphi_{\mathrm{app}} \left(C\alpha\right) & =  \sup_{\tau \geq 0}\left[\psi_{u,s,a}\left(\tau\right) - \frac{\tau}{C\alpha}\right]\\
& = \sup_{\tau' \geq 0}\left[\psi_{u,s,a}\left(C^{\frac{a+u}{a+s}}\tau'\right) - C^{\frac{u-s}{a+s}}\frac{\tau'}{\alpha}\right]\\
& \leq \sup_{\tau' \geq 0}\left[C^{\frac{u-s}{a+s}}\psi_{u,s,a}\left(\tau'\right) - C^{\frac{u-s}{a+s}}\frac{\tau'}{\alpha}\right]\\
& = C^{\frac{u-s}{a+s}} \varphi_{\mathrm{app}} \left(\alpha\right)
\end{align*}
for $C > 1$, i.e. we have
\begin{equation}\label{eq:constant_out}
\varphi_{\mathrm{app}}\left(C\alpha\right) \leq \max\left\{1, C^{\frac{u-s}{a+s}}\right\}\varphi_{\mathrm{app}}\left(\alpha\right)
\end{equation}
for all $\alpha, C>0$.
\item[(d)] Fix $\alpha>0$. By the equality condition in the Fenchel-Young inequality \eqref{eq:fenchel_young} it holds
\begin{equation}\label{eq:young_equality}
\varphi_{\mathrm{app}} \left(\alpha\right) = \psi_{u,s,a} \left(\tau\left(\alpha\right)\right) - \frac{\tau\left(\alpha\right)}{\alpha}
\end{equation}
for any choice $\tau\left(\alpha\right) \in \partial \left(- \psi_{u,s,a} \right)^* \left(-\frac{1}{\alpha}\right)$. Employing \eqref{eq:concavity_psiusa} we find
\[
0 \leq \varphi_{\mathrm{app}} \left(\alpha\right) = \psi_{u,s,a} \left(\tau\left(\alpha\right)\right) - \frac{\tau\left(\alpha\right)}{\alpha} \leq \max\left\{1, \left(\tau\left(\alpha\right)\right)^{\frac{u-s}{a+u}}\right\} \psi_{u,s,a} \left(1\right) - \frac{\tau\left(\alpha\right)}{\alpha},
\]
which implies
\[
\frac{\tau\left(\alpha\right)}{\max\left\{1, \left(\tau\left(\alpha\right)\right)^{\frac{u-s}{a+u}}\right\}} \leq \psi_{u,s,a} \left(1\right)\alpha.
\]
As $u-s < a+u$ this yields $\tau\left(\alpha\right) \to 0$ as $\alpha \to 0$ and furthermore by \eqref{eq:young_equality} that $\varphi_{\mathrm{app}} \left(\alpha\right) \to 0$ as $\alpha \to 0$.
\end{enumerate}
\end{rem}

\begin{cor}\label{cor:rates}
Let the assumptions of Theorem \ref{thm:deterministc} hold true, suppose that $\psi_{u,s,a}$ is concave, and let $\alpha = \alpha_*$ be chosen such that
\begin{equation}\label{eq:apriori_choice}
-\frac{1}{\alpha^*} \in \partial \left(-\psi_{u,s,a}\right) \left(\delta^2\right).
\end{equation}
Then we obtain the convergence rate
\begin{equation} \label{eq:convrate1}
\|f_{\alpha_*}^\delta-\fdag\|_s = \mathcal O \left(\sqrt{\psi_{u,s,a}\left(\delta^2\right)}\right) = \mathcal O \left( \left(\varphi\left(\delta\right)\right)^{\frac{u-s}{a+u}}\right)\qquad\text{as}\qquad\delta \to 0.
\end{equation}
\end{cor}
\begin{proof}
Due to Remark \ref{rem:app_err}(c) we can simplify the error estimate \eqref{eq:errorest1} to
\begin{equation}\label{eq:aux20}
\norm{\fad - \fdag}{s}^2 \leq C' \left(\frac{\delta^2}{\alpha} + \left(- \psi_{u,s,a}\right)^* \left(-\frac{1}{\alpha}\right)\right)
\end{equation}
with $C' = \max\left\{C, C\left(8C\right)^{\frac{u-s}{a+s}}\right\}$ and $C$ as in Theorem \ref{thm:deterministc}. Note that the infimum over $\alpha >0$ of the right-hand side of \eqref{eq:aux20} can be computed as
\begin{align*}
\inf_{\alpha>0} \left[\frac{\delta^2}{\alpha} + \left(- \psi_{u,s,a}\right)^* \left(-\frac{1}{\alpha}\right)\right]& = - \sup_{\alpha>0} \left[-\frac{\delta^2}{\alpha} - \left(- \psi_{u,s,a}\right)^* \left(-\frac{1}{\alpha}\right)\right]\\
& =  - \sup_{\tau<0} \left[\delta^2 \tau - \left(- \psi_{u,s,a}\right)^* \left(\tau\right)\right]\\
& = - \left(- \psi_{u,s,a}\right)^{**} \left(\delta^2\right).
\end{align*}
By concavity of $\psi_{u,s,a}$, the last expression equals $\psi_{u,s,a}\left(\delta^2\right)$. Furthermore choosing $\alpha = \alpha^*$ such that the infimum is attained at $\alpha^*$ corresponds to equality in the Fenchel-Young inequality
\[
-\frac{\delta^2}{\alpha} \leq \left(- \psi_{u,s,a}\right)^* \left(-\frac{1}{\alpha}\right) + \left(- \psi_{u,s,a}\right)^{**} \left(\delta^2\right),
\]
which is attained if and only if $-\frac{1}{\alpha^*} \in \partial \left(-\psi_{u,s,a}\right) \left(\delta^2\right)$. It remains to show that $\alpha_*$ as in \eqref{eq:apriori_choice} satisfies \eqref{eq:conv}. By the equality condition in the Fenchel-Young inequality \eqref{eq:fenchel_young} it holds
\[
\left(-\psi_{u,s,a}\right)^* \left(-\frac{1}{\alpha^*}\right) = \psi_{u,s,a} \left(\delta^2\right) - \frac{\delta^2}{\alpha^*}.
\]
As the left-hand side is $\geq 0$, this implies immediately $\frac{\delta^2}{\alpha^*} \leq \psi_{u,s,a} \left(\delta^2\right) \to 0$ as $\delta\to 0$.

For any convex function on $\left[0,\infty\right)$, the subdifferential can be represented as an interval with borders given by left- and right-hand sided derivatives. Thus the concavity of $\psi_{u,s,a}$ implies
\[
-\partial \left(-\psi_{u,s,a}\right)  \left(\delta^2\right) = \left[\sup_{t \in \left(\delta^2, \infty\right)} \frac{\psi_{u,s,a} \left(t\right) - \psi_{u,s,a}\left(\delta^2\right)}{t-\delta^2}, \inf_{t \in \left[0,\delta^2\right)} \frac{\psi_{u,s,a}\left(\delta^2\right)  - \psi_{u,s,a}\left(t\right)}{\delta^2-t} \right]
\]
As the supremum tends to $\infty$ as $\delta$ tends to $0$ (c.f. \cite[Rem 3.31]{Werner12}), this also proves $\alpha \to 0$ as $\delta \to 0$.
\end{proof}

\begin{rem}
\begin{enumerate}
\item[(a)] The additional assumption that $\psi_{u,s,a}$ itself is also concave in Corollary \ref{cor:rates} seems rather mild. In case of a Hölder-type function $\varphi$, this follows immediately from concavity of $\varphi$ itself, see Example \ref{ex:Hoelder} below. Similarly, if $\varphi$ is of logarithmic type as in Example \ref{ex:stab}, then concavity of $\psi_{u,s,a}$ is also evident.
\item[(b)] We will give another possible expression for an a priori parameter choice rule avoiding convex analysis in Corollary \ref{cor:stoch}.
\end{enumerate}
\end{rem}

Let us now turn to an a posteriori parameter choice rule. Given a set of candidate parameters  $\alpha_1 = \delta^2$, $\alpha_j = \alpha_1 r^{2j-2}$ with some $r>1$ for $j = 2,...,m$ where $m$ is the first value such that $\alpha_m \geq 1$, we define
\begin{equation}\label{eq:lepskij}
j_{\mathrm{Lep}} = \max\left\{1 \leq j \leq m ~\big|~ \norm{\hat u_{\alpha_i} -  \hat u_{\alpha_j}}{s} \leq 4 r^{1-i} \quad\text{for all}\quad i \leq j \right\},
\end{equation}
i.e. $\alpha_{\mathrm{Lep}} := \alpha_{j_{\mathrm{Lep}}}$ is chosen according to the Lepski{\u\i}-type balancing principle. This gives the following result:
\begin{cor}\label{cor:apost}
Let the assumptions of Theorem \ref{thm:deterministc} hold true and choose $\alpha = \alpha_{\mathrm{Lep}}$ according to \eqref{eq:lepskij}. Suppose further that $f_{\alpha}^\delta \in Q$ for all $\alpha$ in the previously described candidate set and any sufficiently small $\delta$ with the set of admissible elements $Q$ for \eqref{eq:stability}. Then we obtain the a posteriori convergence rate
\begin{equation} \label{eq:convrate2}
\|f_{\alpha_{\mathrm{Lep}}}^\delta-\fdag\|_s = \mathcal O \left(\sqrt{\psi_{u,s,a}\left(\delta^2\right)}\right) = \mathcal O \left( \left(\varphi\left(\delta\right)\right)^{\frac{u-s}{a+u}}\right)\qquad\text{as}\qquad\delta \to 0.
\end{equation}
\end{cor}
\begin{proof}
Note that \eqref{eq:errorest1} together with \eqref{eq:constant_out} yields an error decomposition of the form
\begin{equation}\label{eq:error_decomposition_det2}
\norm{\fad-\fdag}{s} \leq \frac{\delta}{\sqrt{\alpha}} + C\sqrt{\varphi_{\mathrm{app}} \left(\alpha\right)}
\end{equation}
with some constant $C>0$. For our set of parameter candidates this gives yields
\[
\norm{\hat u_{\alpha_j}-\fdag}{s} \leq \frac12 \left(\Phi\left(j\right) + \Psi\left(j\right)\right), \qquad 1\leq j \leq m
\]
with $\Psi\left(j\right) = (2\delta)/\sqrt{\alpha_j} = 2 r^{1-j}$ and $\Phi\left(j\right) = C\sqrt{\varphi_{\mathrm{app}} \left(\alpha_j\right)}$. By construction, $\Psi$ is non-increasing, $\Phi$ is non-decreasing, and $\Phi\left(1\right) \leq \Psi\left(1\right) = 2$ if $\delta$ is sufficiently small. 
Furthermore $\Psi\left(i\right) \leq r \Psi\left(i+1\right)$, and hence it follows from \cite[Cor. 1]{Mathe06} that
\begin{align*}
\norm{\hat u_{\alpha_{j_{\mathrm{Lep}}}}-\fdag}{s} &\leq 3 r \min_{1 \leq j \leq m} \left[\Phi\left(j\right) + \Psi\left(j\right)\right]\\
& = 3 r C\min_{1 \leq j \leq m} \left[\frac{\delta}{2 \sqrt{\alpha_j}} + \sqrt{\left(-\psi_{u,s,a}\right)^*\left(-\frac{1}{\alpha_j}\right)}\right].
\end{align*}
By some elementary convex analysis we conclude as in \cite[Lemma 3.42]{Werner12}, exploiting \eqref{eq:constant_out}, that the minimum on the right-hand side can be replaced by the infimum over all $\alpha$, provided that $\delta$ is sufficiently small (cf. also \cite{WernerHohage12}).
\end{proof}

\begin{ex}[H\"older type conditional stability] \label{ex:Hoelder}
Let us consider the H\"older special case $\varphi\left(t\right) = t^\gamma$ of the conditional stability estimate (\ref{eq:stability})  with exponents $0<\gamma \le 1$, which has recently been studied in a slightly
modified form in~\cite{EggerHof18}. Here we obtain
\[
\psi_{u,s,a}\left(t\right) = t^{\frac{\gamma\left(u-s\right)}{a+u}}
\]
and hence for $q := \frac{\gamma\left(u-s\right)}{a+u}$ that
\[
\left(-\psi_{u,s,a}^*\right) \left(-v\right) = \sup_{t \geq 0} \left[t^q - tv\right] \sim v^{\frac{q}{q-1}},
\]
because it can be seen via differentiation that the supremum is attained for $t = \left(\frac{v}{q}\right)^{\frac{1}{q-1}}$.
Thus $\varphi_{\mathrm{app}} \left(\alpha\right) \sim \alpha^{\frac{q}{1-q}}$ and
\[
\alpha\left(-\psi_{u,s,a}^*\right) \left(-\frac{1}{8\alpha}\right) \sim \alpha^{\frac{q}{1-q} + 1} \sim \alpha^{\frac{1}{1-q}} \sim \alpha^{\frac{a+u}{a+u-\gamma\left(u-s\right)}}.
\]
This term coincides with the corresponding error term in \cite[Lemma~3.3]{EggerHof18}. Hence, the convergence rate from (\ref{eq:convrate1}) attains in this example the form
\begin{equation} \label{eq:rateex}
\|f^\delta_{\alpha_*}-\fdag\|_{s}  = \mathcal O \left(\delta^{\gamma \frac{u-s}{a+u}}\right)\qquad\text{as}\qquad\delta \to 0,
\end{equation}
which again coincides with the rate results of Theorems~2.1 and 2.2 in~\cite{EggerHof18}. Note that in case of a linear forward operator, these rates are known to be order optimal as also discussed in~\cite{EggerHof18}. The a priori choice \eqref{eq:apriori_choice} for the regularization parameter leads for the H\"older type conditional stability to
\begin{equation} \label{eq:apriex}
\alpha_*=\alpha_*(\delta) \sim \delta^{2-2\gamma\frac{u-s}{a+u}}.
\end{equation}
\end{ex}

\begin{rem} \label{rem:border}
The case of H\"older type conditional stability considered in Example~\ref{ex:Hoelder} allows us to discuss briefly the borderline situation $u=s$. Evidently, then the a priori parameter choice \eqref{eq:apriex}
attains the form $\alpha_*=\alpha_*(\delta) \sim \delta^2$, which is in a general Hilbert space setting well-known from \cite{ChengYamamoto00} as an appropriate choice for conditional stability estimates of a form like in Example~\ref{ex:Q3} below. However, in our setting $\fdag \in \X_s$ with this parameter choice formula \eqref{eq:rateex} cannot serve as a convergence rate result, because the exponent of $\delta$ is not positive. Moreover Proposition~\ref{pro:convergence} does not apply, since $\delta^2/\alpha_* \to 0$ as $\delta \to 0$ fails. Hence, one cannot even show at all convergence $\|f^\delta_{\alpha_*}-\fdag\|_{s} \to 0$ as $\delta \to 0$ and if the set $Q$ in (\ref{eq:stability}) restricts the applicability of the conditional stability estimate to balls around $\fdag$, then $u=s$ is in contrast to $u>s$ does not ensure that
$f^\delta_{\alpha_*} \in Q$.  Asking for reasons why \cite{ChengYamamoto00} recommends
$\alpha_*=\alpha_*(\delta) \sim \delta^2$ nevertheless also for the borderline situation $u=s$ of conditional stability, we see that Cheng and Yamamoto in \cite{ChengYamamoto00} use for finding approximate solutions the minimization problem
$$\falhat \in \argmin\limits_{f\in D(F)\cap Q} \left[ \frac12 \norm{F\left(f\right) - \gobs}{\Y}^2 + \alpha \norm{f}{s}^2 \right]$$
instead of \eqref{eq:tik}, which needs to know the set $Q$. Then one can show at least a convergence rate result in the $\X$-norm of the form
\begin{equation} \label{eq:ratezero}
\|f^\delta_{\alpha_*}-\fdag\|  = \mathcal O \left(\delta^{\gamma \frac{u}{a+u}}\right)\qquad\text{as}\qquad\delta \to 0.
\end{equation}
For $\gamma=1$ such rate result \eqref{eq:ratezero} takes place also under somewhat stronger conditions for \linebreak `oversmoothing' penalties in the case $u<s$ with $\|\fdag\|_s=\infty$. In this context, we refer to \cite{HofMat19}, where for the a priori parameter choice \eqref{eq:apriex}, here with $\delta^2/\alpha_* \to \infty$ as $\delta \to 0$, \eqref{eq:ratezero} is proven, see also \cite{HofMat18} for the same
convergence rate result by using the discrepancy principle.

\end{rem}

\section{Statistical Inverse Problems} \label{sec:Stat}

Now we will discuss how to generalize the previous results to the stochastic data model \eqref{eq:stochastic_noise_model} with $\sigma>0$. To analyze \eqref{eq:tik_stat} we have to proceed differently and post additional assumptions:
\begin{ass}\label{ass:stoch}
Let us assume that there is a Gelfand triple $\left(\V,\Y,\V'\right)$ such that the embedding $\iota: \V \hookrightarrow \Y$ is Hilbert-Schmidt. Furthermore we suppose that $F$ satisfies the interpolation inequality
\begin{equation}\label{eq:interpolationY}
\norm{F\left(f\right) - \gdag}{\V} \leq C_\theta\left(\rho\right) \norm{F\left(f\right) - \gdag}{\Y}^\theta \norm{f - \fdag}{s}^{1-\theta}
\end{equation}
for all $f \in D^s_\rho\left(\fdag\right)$ with some constant $C_\theta \left(\rho\right)$, $\rho > 0$ and $\theta \in \left(0,1\right)$.
\end{ass}
This assumption requires some comments. First note that $\iota$ being Hilbert-Schmidt implies
\[
\mathbb E\left[\norm{Z}{\V'}^2\right] = \text{trace}\left(\iota^* \text{Cov}\left[Z\right]\iota\right) < \infty,
\]
i.e. it holds $\norm{Z}{\V'} < \infty$ a.s.

\begin{rem}
Suppose that the Gelfand triple $\left(\V,\Y,\V'\right)$ is part of a Hilbert scale $\left\{\Y_\mu\right\}_{\mu \in \R}$, i.e. there exists $t \in \R$ such that $\V = \Y_t$, $\V' = \Y_{-t}$ and $\Y = \Y_0$. Furthermore assume that $F$ is Lipschitz continuous as $F : \X_s \to \Y_r$ for some $r > t$. Then \eqref{eq:interpolationY} is satisfied.
\end{rem}
\begin{proof}
The interpolation inequality \eqref{eq:interpolationX} for the Hilbert scale $\left\{\Y_\mu\right\}_{\mu \in \R}$ yields
\[
\norm{g}{\V} \leq \norm{g}{\Y}^\theta \norm{g}{\Y_r}^{1-\theta}
\]
with $\theta = 1-t/r$. Consequently, we find
\begin{align*}
\norm{F\left(f\right) - \gdag}{\V} &\leq \norm{F\left(f\right) - \gdag}{\Y}^\theta \norm{F\left(f\right) - \gdag}{\Y_r}^{1-\theta} \\
& \leq L^{1-\theta}\norm{F\left(f\right) - \gdag}{\Y}^\theta \norm{f - \fdag}{s}^{1-\theta}
\end{align*}
with the Lipschitz constant $L$ of $F : \X_s \to \Y_r$.
\end{proof}	
\begin{ex}
The most common example for white noise $\xi$ is as follows. Let $\Y = L^2 \left(\Omega\right)$ for some Lipschitz domain $\Omega \subset \R^d$, and let $\V = H^s\left(\Omega\right)$ with some $s > \frac{d}{2}$. Then $\iota :  H^s\left(\Omega\right) \hookrightarrow L^2 \left(\Omega\right)$ is Hilbert-Schmidt and one has the interpolation inequality (cf. \eqref{eq:interpolationX})
\[
\left\Vert g\right\Vert_{H^t} \leq \left\Vert g \right\Vert_{L^2}^\theta \left\Vert g \right\Vert_{H^r}^{1-\theta}
\]
with $\theta = 1-t/r$ whenever $r > t$. Consequently, if $F : \X_s \to H^r\left(\Omega\right)$ is Lipschitz, then \eqref{eq:interpolationY} holds true.
\end{ex}

It follows similar to the deterministic case that the functional \eqref{eq:tik_stat} admits a unique minimizer for fixed data $\gobs$. If $Z$ is considered as an element of $\Y^*$, then continuous dependency of $\falhat$ on $Z$ can also be shown following the deterministic results. Convergence and convergence rates are slightly more involved, as we will see below. For the sake of presentation we restrict ourselves to a convergence rates result:

\begin{thm}\label{thm:stochastic}
Let the Assumptions \ref{ass:basic1}, \ref{ass:basic2} and \ref{ass:stoch} be satisfied, let the data $\gobs$ be given as in \eqref{eq:stochastic_noise_model}, and suppose \eqref{eq:interpolationY} holds true. If there are $\sigma_0, \delta_0>0$ and $\alpha$ is chosen such that $\falhat \in D^s_\rho\left(\fdag\right)$ for all $0 < \sigma \leq \sigma_0$, $0 < \delta \leq \delta_0$ (with $\rho$ as in Assumption \ref{ass:stoch}), then we have (surely) the error estimate
\[
\frac18 \norm{F\left(\falhat\right)-\gdag}{\Y}^2 + \frac{\alpha}{4} \norm{\falhat-\fdag}{s}^2 \leq C \left[\sigma^2 \norm{Z}{\V'}^2\alpha^{\theta-1} + \delta^2 + \alpha \left(-\psi_{u,s,a}\right)^* \left(-\frac{1}{8 C\alpha}\right)\right]
\]
for some constant $C>0$.
\end{thm}
\begin{proof}
Denote again
\[
\S{g} := \frac12 \norm{g}{\Y}^2 - \left\langle g, \gobs\right\rangle \qquad\text{and}\qquad \T{g} := \frac12\norm{g-\gdag}{\Y}^2
\]
for $g \in \Y$. Due to the minimizing property of $\falhat$ in \eqref{eq:tik_stat} we have
\begin{align*}
\S{F\left(\falhat\right)} + \alpha \norm{\falhat}{s}^2  \leq  \S{F\left(\fdag\right)} + \alpha \norm{\fdag}{s}^2,
\end{align*}
which combined with \eqref{eq:aux1} implies that
\begin{align}
\frac12\T{F\left(\falhat\right)} + \frac{\alpha}{s} \norm{\falhat-\fdag}{2}^2 \leq& \frac12\T{F\left(\falhat\right)} - \left(\S{F\left(\falhat\right)} - \S{\gdag}\right)\nonumber\\ & + C \alpha \varphi\left(\norm{F\left(\falhat\right) - \gdag}{\Y}\right)^{\frac{2u-2s}{s+a}}\label{eq:aux4}.
\end{align}
By definition of $\S{\cdot}$ and $\T{\cdot}$ we obtain
\begin{align*}
&\frac12\T{F\left(\falhat\right)} - \left(\S{F\left(\falhat\right)} - \S{\gdag}\right) \\
=& \frac14 \norm{F\left(\falhat\right)-\gdag}{\Y}^2 - \left(\frac12 \norm{F\left(\falhat\right)}{\Y}^2 - \left\langle F \left(\falhat\right), \gobs\right\rangle - \frac12\norm{\gdag}{\Y}^2 + \left\langle \gdag, \gobs\right\rangle\right)\\
=& \frac14 \norm{F\left(\falhat\right)-\gdag}{\Y}^2 - \left(\frac12 \norm{F\left(\falhat\right)}{\Y}^2 - \left\langle F\left(\falhat\right),\gdag\right\rangle - \left\langle F\left(\falhat\right)-\gdag,\sigma Z + \delta \xi\right\rangle + \frac12 \norm{\gdag}{\Y}^2\right)\\
& = \frac14 \norm{F\left(\falhat\right)-\gdag}{\Y}^2 - \left(\frac12 \norm{F\left(\falhat\right)-\gdag}{\Y}^2- \left\langle F\left(\falhat\right)-\gdag,\sigma Z + \delta \xi\right\rangle\right)\\
& = - \frac14 \norm{F\left(\falhat\right)-\gdag}{\Y}^2 + \sigma \left\langle F\left(\falhat\right)-\gdag, Z\right\rangle + \delta\left\langle F\left(\falhat\right)-\gdag, \xi\right\rangle\\
& \leq - \frac14 \norm{F\left(\falhat\right)-\gdag}{\Y}^2 + \sigma \norm{F\left(\falhat\right)-\gdag}{\V}\norm{Z}{\V'} + \delta \norm{F\left(\falhat\right)-\gdag}{\Y}.
\end{align*}
For the last term on the right-hand side we use $ab \leq 2 a^2 + \frac18 b^2$, which yields the estimate
\begin{align*}
&\frac12\T{F\left(\falhat\right)} - \left(\S{F\left(\falhat\right)} - \S{\gdag}\right) \\
& \leq - \frac18 \norm{F\left(\falhat\right)-\gdag}{\Y}^2 + \sigma \norm{F\left(\falhat\right)-\gdag}{\V}\norm{Z}{\V'} + 2\delta^2.
\end{align*}
Concerning the second term on the right-hand side, using \eqref{eq:interpolationY} and applying \eqref{eq:young} appropriately twice we obtain
\begin{align*}
&\frac12\T{F\left(\falhat\right)} - \left(\S{F\left(\falhat\right)} - \S{\gdag}\right) \\
\leq& - \frac18 \norm{F\left(\falhat\right)-\gdag}{\Y}^2 + C\sigma\norm{Z}{\V'} \norm{F\left(\falhat\right) - \gdag}{\Y}^\theta \norm{\falhat - \fdag}{s}^{1-\theta}+ 2\delta^2 \\
\leq & C' \left(\sigma \norm{Z}{\V'}\norm{\falhat - \fdag}{s}^{1-\theta}\right)^{\frac{2}{2-\theta}} + 2\delta^2 \\
\leq & C'' \sigma^2 \norm{Z}{\V'}^2\alpha^{\theta-1} + \frac{\alpha}{4} \norm{\falhat - \fdag}{s}^2+ 2\delta^2
\end{align*}
with some constants $C,C',C''>0$ as $\norm{\falhat}{s}$ is bounded. Altogether this yields
\[
\frac12\T{F\left(\falhat\right)} + \frac{\alpha}{4} \norm{\falhat-\fdag}{s}^2 \leq C \left[\sigma^2 \norm{Z}{\V'}^2\alpha^{\theta-1} + \delta^2 + \alpha \varphi\left(\norm{F\left(\falhat\right) - \gdag}{\Y}\right)^{\frac{2u-2s}{s+a}}\right]
\]
with some generic constant $C>0$. Now we can proceed as in the deterministic case.
\end{proof}

\begin{cor}\label{cor:stoch}
Let the assumptions of Theorem \ref{thm:stochastic} be satisfied and recall the notation
\[
\varphi_{\mathrm{app}} \left(\alpha\right)= \left(-\psi_{u,s,a}\right)^* \left(-\frac{1}{\alpha}\right), \qquad\alpha > 0.
\]
Define
\[
\Sigma \left(\alpha\right) = \sqrt{\alpha} \sqrt{\varphi_{\mathrm{app}}\left(\alpha\right)} \qquad\text{and}\qquad \widetilde{\Sigma}\left(\alpha\right) = \alpha^{1-\frac{\theta}{2}} \sqrt{\varphi_{\mathrm{app}}\left(\alpha\right)}, \quad\alpha > 0
\]
and choose $\alpha$ such that
\begin{equation}\label{eq:apriori_choice2}
\alpha\sim \left(\Sigma^{-1} \left(\delta\right) + \widetilde{\Sigma}^{-1} \left(\sigma\right)\right)\qquad\text{as}\qquad \max\left\{\delta,\sigma\right\} \to 0.
\end{equation}
Then we obtain the a.s. convergence rate
\[
\norm{\falhat - \fdag}{s} = \mathcal O \left(\sqrt{\varphi_{\mathrm{app}}\left(\Sigma^{-1} \left(\delta\right) + \widetilde{\Sigma}^{-1} \left(\sigma\right)\right)}\right)
\]
as $\max\left\{\delta,\sigma\right\} \to 0$.
\end{cor}
\begin{proof}
According to Theorem \ref{thm:stochastic} we have
\[
\frac{1}{C} \norm{\falhat-\fdag}{s}^2 \leq \sigma^2 \alpha^{\theta-2} + \frac{\delta^2}{\alpha} + \varphi_{\mathrm{app}} \left(\alpha\right)
\]
a.s. for some sufficiently large $C> 0$, where we also exploited $\norm{Z}{\V'} < \infty$ a.s. and Remark \ref{rem:app_err}(c). Via $\delta = \Sigma\left(\Sigma^{-1}\left(\delta\right)\right) = \sqrt{\Sigma^{-1}\left(\delta\right)} \sqrt{\varphi_{\mathrm{app}}\left(\Sigma^{-1}\left(\delta\right)\right)}$ and analogously $\sigma = \left(\widetilde{\Sigma}^{-1} \left(\sigma\right)\right)^{1- \frac{\theta}{2}} \sqrt{\varphi_{\mathrm{app}} \left(\widetilde{\Sigma}^{-1} \left(\sigma\right)\right)}$ we hence obtain
\begin{align*}
\norm{\falhat-\fdag}{s} &\lesssim \frac{\sigma}{\alpha^{1-\frac{\theta}{2}}} + \frac{\delta}{\sqrt{\alpha}} + \sqrt{\varphi_{\mathrm{app}} \left(\alpha\right)} \\
& \lesssim \frac{\sigma}{\left(\widetilde{\Sigma}^{-1} \left(\sigma\right)\right)^{1-\frac{\theta}{2}}} + \frac{\delta}{\sqrt{\Sigma^{-1}\left(\delta\right)}} + \sqrt{\varphi_{\mathrm{app}} \left(\alpha\right)} \\
&=  \sqrt{\varphi_{\mathrm{app}}\left(\Sigma^{-1}\left(\delta\right)\right)} + \sqrt{\varphi_{\mathrm{app}} \left(\widetilde{\Sigma}^{-1} \left(\sigma\right)\right)} + \sqrt{\varphi_{\mathrm{app}} \left(\alpha\right)}\\
& \lesssim \sqrt{\varphi_{\mathrm{app}}\left(\Sigma^{-1} \left(\delta\right) + \widetilde{\Sigma}^{-1} \left(\sigma\right)\right)}
\end{align*}
a.s., where $\lesssim$ means up to a multiplicative constant which can change from line to line, but is independent of $\alpha, \sigma$ and $\delta$.
\end{proof}

\begin{rem}
\begin{enumerate}
\item[(a)] It is immediately be clear that the convergence rate in Corollary \ref{cor:stoch} can also be obtained under an a posteriori choice of $\alpha$ as in Corollary \ref{cor:apost}.
\item[(b)] In the case $\varphi\left(t\right) = t^\gamma$ with some $0 < \gamma \leq 1$ as discussed in Example \ref{ex:Hoelder}, we compute
\[
\Sigma\left(\alpha\right) = \alpha^{\frac{1}{2\left(1-q\right)}}, \qquad q = \gamma \frac{u-s}{a+u}
\]
and hence it can be seen immediately that the a priori choices in \eqref{eq:apriori_choice} and \eqref{eq:apriori_choice2} and also the obtained rates in Theorem \ref{thm:deterministc} and Theorem \ref{thm:stochastic} with $\sigma = 0$ coincide.
\end{enumerate}
\end{rem}

Concerning the assumptions of Theorem \ref{thm:stochastic}, we finally  mention the following:
\begin{rem}
Suppose that $F$  maps locally Lipschitz continuous from $\X_s$ into $\V$, i.e. there is some $C = C \left(\rho\right)$ such that
\[
\norm{F\left(f_1\right) - F\left(f_2\right)}{\V} \leq C \left(\rho\right) \norm{f_1-f_2}{s}
\]
for all $f_1, f_2 \in D_\rho^s \left(0\right)$, and that $C \left(\rho\right) = o \left(\rho\right)$. Then any parameter choice $\alpha = \alpha_*$ such that
\[
\max\left\{ \frac{\sigma}{\alpha}, \frac{\delta^2}{\alpha}\right\} \to 0
\]
as $\sigma, \delta \to 0$ yields $\falhat \in D^s_{\bar{\rho}}\left(0\right)$ a.s. with a suitable $\bar\rho>0$ as $\sigma, \delta \to 0$.
\end{rem}
\begin{proof}
Similar to the proof of Theorem \ref{thm:stochastic} we obtain from the minimizing property
\begin{align*}
\alpha \norm{\falhat}{s}^2  &\leq  \S{F\left(\fdag\right)} - \S{\gdag} + \alpha \norm{\fdag}{s}^2 \\
& = -\frac12 \norm{F\left(\falhat\right) - \gdag}{\Y}^2 + \left\langle F\left(\falhat\right) - \gdag, \sigma Z + \delta \xi\right\rangle+ \alpha \norm{\fdag}{s}^2\\
& \leq \frac{\delta^2}{2} + \sigma \norm{Z}{\V'} \norm{F\left(\falhat\right) - \gdag}{\V} + \alpha \norm{\fdag}{s}^2\\
& \leq \frac{\delta^2}{2} + \sigma \norm{Z}{\V'} C \left(\max\left\{\norm{\falhat}{s}, \norm{\fdag}{s}\right\}\right) + \alpha \norm{\fdag}{s}^2.
\end{align*}
As $\norm{Z}{\V'}$ is a.s. bounded, this implies by $C\left(\rho\right) = o \left(\rho\right)$ the claim.
\end{proof}

\section*{Appendix}

In this appendix we discuss different approaches to derive conditional stability estimates as in Assumption \ref{ass:basic2}.

\subsection*{Variant (A): based on local structural conditions on the nonlinearity of $F$}

\begin{ex}[$Q=D_r(\fdag),\;a>0,\,\theta=0,$ strong nonlinearity conditions of tangential cone type] \label{ex:Q1}
This situation assumes that the forward operator $F$ is G\^ateaux or Fr\'echet differentiable at $\fdag$ with the derivative $F^\prime(\fdag) \in \mathcal{L}(\X,\Y)$. Moreover, it is
characterized by the pair of conditions
\begin{equation} \label{eq:illposednessdegree}
\|h\|_{-a} \le \bar K \,\|F^\prime(\fdag)\,h\|_{\Y} \qquad \mbox{for all} \quad h \in X
\end{equation}
and
\begin{equation} \label{eq:generalnonlinearity}
\|F^\prime(\fdag)(f-\fdag)\|_{\Y} \le \tilde K\, \varphi(\|F(f)-F(\fdag)\|_{\Y})  \qquad \mbox{for all} \quad f \in D_r(\fdag),
\end{equation}
where $\varphi$ is a concave index function and $\bar K,\tilde K$ as well as $r$ are positive constants. The first condition \eqref{eq:illposednessdegree} often occurs in regularization literature for Hilbert scale models (cf., e.g.,~\cite{Neubauer92,Neubauer00,Tautenhahn94,Tautenhahn98}) , sometimes also in the stronger version $\|h\|_{-a} \sim \|F^\prime(\fdag)\,h\|_{\Y}$ for all $h \in \X$, where $a>0$ denotes the \textit{degree of ill-posedness} locally at $\fdag$ (cf.~\cite[Sec.~10.4]{EnglHankeNeubauer96}). In the form with a general concave index function $\varphi$, the second condition \eqref{eq:generalnonlinearity} was introduced and exploited in \cite{BoHo10}.
In the special case of monomials $\varphi(t)=t^\kappa$, however, with exponents $0<\kappa \le 1$ and associated with H\"older rates this condition plays some role in the context of the \textit{degree} $(\kappa,\zeta)$ \textit{of nonlinearity} of $F$ at $\fdag$ introduced in \cite{HofSch94}, where the inequality
\begin{equation} \label{eq:degreenonlinearity}
\|F(f)-F(\fdag)-F^\prime(\fdag)(f-\fdag)\|_{\Y} \le \hat K\,\|F(f)-F(\fdag)\|^\kappa_{\Y}\,\|f-\fdag\|^\zeta  \qquad \mbox{for all} \quad f \in D_r(\fdag)
\end{equation}
for exponents $0 \le\kappa \le 1$, $\;0\le \zeta \le 2\,$ and a constant $\hat K>0$ has been considered.

For strong structural conditions of nonlinearity of interest in this example and in particular for condition \eqref{eq:generalnonlinearity},
exponents $\kappa>0$ are required in \eqref{eq:degreenonlinearity}. Evidently, by the triangle inequality
$$\|F(f)-F(\fdag)-F^\prime(\fdag)(f-\fdag)\|_{\Y} \le  \|F^\prime(\fdag)(f-\fdag)\|_{\Y}+\|F(f)-F(\fdag)\|_{\Y} $$
we obtain from  \eqref{eq:generalnonlinearity} with $\varphi(t)=t^\kappa\;(0<\kappa \le 1)$ a \textit{tangential cone} type condition
\begin{equation} \label{eq:nlkappa}
\|F(f)-F(\fdag)-F^\prime(\fdag)(f-\fdag)\|_{\Y} \le \hat {\bar K}\,\|F(f)-F(\fdag)\|^\kappa_{\Y}  \qquad \mbox{for all} \quad f \in D_r(\fdag)
\end{equation}
with a constant $\hat{\bar K}>0$ depending on $\tilde K$ and $r$.
Vice versa, we also derive by the triangle inequality a condition  \eqref{eq:generalnonlinearity} with $\varphi(t)=t^\kappa\;(0<\kappa \le 1)$
from  \eqref{eq:nlkappa}.

Weaker structural conditions of nonlinearity, which will be discussed in Example~\ref{ex:Q2} below, are characterized by the fact that $F$ at $\fdag$ does not allow for exponents $\kappa>0$ in \eqref{eq:degreenonlinearity},
but exponents $\kappa=0$ and $1<\zeta \le 2$ are typical in case of a H\"older continuity of the derivative $F^\prime(f)$ in a neighborhood of $\fdag$.

Now we come back to the pair \eqref{eq:illposednessdegree} and \eqref{eq:generalnonlinearity} of conditions and derive for this situation the corresponding structure of the set $Q$ in Assumption~\ref{ass:basic2}.
Combining both inequalities we immediately find a conditional stability estimate \eqref{eq:stability} of the form
\begin{equation} \label{eq:strongdegree}
\|f-\fdag\|_{-a} \le R \,\varphi(\|F(f)-F(\fdag)\|_{\Y})  \qquad \mbox{for all} \quad f \in D_r(\fdag)
\end{equation}
with some constant $R>0$ and for the subset $Q=D_r(\fdag)$ of the closed intersected ball $D_\rho^\theta(0)$, where $\theta=0$ and $\rho=\|\fdag\|+r$.

Let us close this example with some special application using $\X=\Y=L^2(0,T)$ and $D(F)=\X$. We consider here the family of forward operators
\begin{equation} \label{eq:exop}
[F(f)](t)=c_0\,\exp\left(c_1\,\int_0^t f(\tau)d\tau \right) \quad (0 \le t \le T)
\end{equation}
with constants $c_0,c_1>0$. Such operators (\ref{eq:exop}) occur in various types of parameter identification problems, e.g.~for finding time-dependent growth rate functions in ordinary differential equation
models and for identifying time-dependent conductivity functions in heat equation models (cf.~for more details \cite{Hof98}).
The corresponding operator equation (\ref{eq:opeq})
is locally ill-posed everywhere on $\X$. Moreover, the operator $F$ is continuously Fr\'echet
differentiable everywhere on $\X$ and its
Fr\'echet derivative attains the form
$$[F^\prime (f) h](t)=c_1\,[F(f)](t) \int_0^t h(\tau)d\tau  \quad (0 \le t \le T,\;\; h \in \X). $$
Furthermore, we have for some constant $\hat K>0$ and for all $f \in \X$
\begin{equation} \label{eq:nldegree}
  \|F(f)-F(\fdag)-F^\prime(\fdag)(f-\fdag)\|_{\Y} \le
  \hat K\,\|F(f)-F(\fdag)\|_{\Y}\,\|f-\fdag\|,
\end{equation}
which indicates without an upper bound for the radius $r>0$ of $D_r(\fdag)$ a degree $(1,1)$ of nonlinearity (cf.~\eqref{eq:degreenonlinearity}).
Applying the triangle inequality to (\ref{eq:nldegree}) yields the estimate
\begin{equation} \label{eq:nldegree1}
 \|F^\prime(\fdag)(f-\fdag)\|_{\Y}\le (\hat K\,\|f-\fdag\|+1)\,\|F(f)-F(\fdag)\|_{\Y} \le \tilde K \,\|F(f)-F(\fdag)\|_{\Y}
\end{equation}
 of the form \eqref{eq:generalnonlinearity} with $\varphi(t)=t$ and valid for all $f \in D_r(\fdag)$, where $\tilde K=r\hat K+1$.
Taking into account the estimate (\ref{eq:nldegree1}) one can consider the integration operator
\begin{equation}\label{eq:J}
[Jh](t):=\int_0^t h(\tau)d\tau \quad (0 \le t \le T)
\end{equation}
mapping in
$\X=L^2(0,T)$ in the context of the related Hilbert scale $\{\X_\tau\}_{\tau \in
  \mathbb{R}}$ generated by the operator $L=(J^*J)^{-1/2}$.
Because there is a constant $0<\underline c < \infty$ with $\underline c \le  [F(\fdag)](t) \;(0 \le t
\le T)$ for the multiplier function in $F^\prime (\fdag)$, we also find a constant $0< c_{down}< \infty$
such that we have the estimate $c_{down} \|f-\fdag\|_{-1} =c_{down} \|J(f-\fdag)\|  \le \|F^\prime(\fdag)(f-\fdag)\|_{\Y}$, which is valid for all $f \in \X$ and can be rewritten in form of an inequality \eqref{eq:illposednessdegree}.
Consequently, we find for all $r>0$ a conditional stability estimate of type (\ref{eq:stability}) with $a=1$ and $\varphi(t)=t$  as
$$\|f-\fdag\|_{-1} \le R\, \|F(f)-F(\fdag)\|_{\Y}\qquad \mbox{for all} \quad f \in D_r(\fdag),$$
where $R=\frac{Kr+1}{c_{down}}$.
\end{ex}

\begin{ex}[$Q=D_r(\fdag)\cap B^\theta_\tau(\fdag),\;a>0,\,\theta>0\,$, weak nonlinearity conditions of H\"older type] \label{ex:Q2}
This situation assumes that the forward operator $F$ is continuously Fr\'echet differentiable in a neighborhood of $\fdag$ with derivatives $F^\prime(f)\in \mathcal{L}(\X,\Y)$ for $f \in D_r(\fdag)$. Moreover,
 it is characterized by the condition
\begin{equation} \label{eq:illposednessdegree2}
\|h\|_{-a} \le \bar K \,\|F^\prime(\fdag)\,h\|_{\Y} \qquad \mbox{for all} \quad h \in X,
\end{equation}
which already occurred in Example~\ref{ex:Q1}, in combination with a local H\"older continuity condition
\begin{equation} \label{eq:generalnonlinearity2}
\|F^\prime(f)-F^\prime(\fdag)\|_{\mathcal{L}(\X,\Y)} \le \check{K}\, \|f-\fdag\|^\eta  \qquad \mbox{for all} \quad f \in D_r(\fdag)
\end{equation}
for the Fr\'echet derivatives, where $\check{K}>0$ is some constant and $0<\eta \le 1$ is the associated H\"older exponent. As it is well-known, one derives immediately from \eqref{eq:generalnonlinearity2} by using
the mean value theorem in integral form the estimate
\begin{equation} \label{eq:Hoeldernonlinearity}
\|F(f)-F(\fdag)-F^\prime(\fdag)(f-\fdag)\|_{\Y} \le \frac{\check{K}}{2}\, \|f-\fdag)\|^{\eta+1}  \qquad \mbox{for all} \quad f \in D_r(\fdag).
\end{equation}
Thus, the degree of nonlinearity $(\kappa,\zeta)$ of $F$ at $\fdag$ (cf.~\eqref{eq:degreenonlinearity}) takes place with $\kappa=0$ and $1<\zeta=\eta+1 \le 2$. As we will see the
condition \eqref{eq:generalnonlinearity2}, which is in some sense weaker than \eqref{eq:nlkappa}, requires in this situation a further restriction of the admissible set $Q$ for obtaining
a conditional stability estimate (\ref{eq:stability}).
\begin{pro} \label{pro:Q2}
Let for $a>0$ and $0<\eta \le 1$ the conditions \eqref{eq:illposednessdegree2} and  \eqref{eq:generalnonlinearity2} hold. Moreover let $\fdag \in \X_{\frac{a}{\eta}} \cap D(F)$. Then
for $\tau<\left(\frac{\bar{K}\,\check{K}}{2}\right)^{-1/\eta}$  the conditional stability estimate
\begin{equation}\label{eq:stability2}
\|f - \fdag\|_{-a} \leq R \, \|F(f)-F(\fdag)\|_{\Y}
\end{equation}
holds for all $f \in D_r(\fdag)\cap D_\tau^{\frac{a}{\eta}}(\fdag)$ with the constant $R=\frac{\bar K}{1-\tilde \tau}$, where we set $\tilde \tau:=\frac{\bar{K}\,\check{K}\tau^\eta}{2}<1$.
\end{pro}
\begin{proof}
From \eqref{eq:illposednessdegree2} and \eqref{eq:Hoeldernonlinearity} we find with the triangle inequality the estimate
$$\|f-\fdag\|_{-a} \le \bar K \,\|F(f)-F(\fdag)\|_{\Y} +\bar K \,\|F(f)-F(\fdag)-F^\prime(\fdag)(f-\fdag)\|_{\Y} $$
$$ \le \bar K \,\|F(f)-F(\fdag)\|_{\Y} + \frac{\bar K\,\check{K}}{2}\, \|f-\fdag\|^{\eta+1} \qquad \mbox{for all} \quad f \in D_r(\fdag). $$
By using the interpolation inequality $\|h\|^{\eta+1} \le \|h\|_{-a}\,\|h\|^\eta_{a/\eta},$ valid for all $h \in \X_{a/\eta},$
we can further estimate for $f \in \X_{a/\eta}$ as
\begin{equation} \label{eq:nearly}
\|f-\fdag\|_{-a} \le \bar K \,\|F(f)-F(\fdag)\|_{\Y} + \left(\frac{\bar K\, \check{K}}{2}\,\|f-\fdag\|^\eta_{a/\eta}\right)\,\|f-\fdag\|_{-a}.
\end{equation}
Taking $f \in D_r(\fdag)\cap D_\tau^{\frac{a}{\eta}}(\fdag)$ with $\tau<\left(\frac{\bar{K}\,\check{K}}{2}\right)^{-1/\eta}$ and setting $\tilde \tau:=\frac{\bar{K}\,\check{K}\tau^\eta}{2}<1$,
the inequality \eqref{eq:nearly} can be rewritten as
$$(1-\tilde \tau)\,\|f-\fdag\|_{-a} \le \bar K\, \|F(f)-F(\fdag)\|_{\Y},$$
which yields the conditional stability estimate \eqref{eq:stability2} and completes the proof.
\end{proof}

We note that the set $Q:= D_r(\fdag)\cap D_\tau^{\frac{a}{\eta}}(\fdag)$ of admissible elements in the conditional stability estimate \eqref{eq:stability2} is a subset of $D_\rho^{\frac{a}{\eta}}(0)$ for
some radius $\rho>0$ depending on $a,\eta,\tau$ and $\fdag$.

\begin{rem} \label{rem:nonsmooth}
An inspection of the above proof shows that \eqref{eq:stability2} is also valid if $\fdag \notin \X_{\frac{a}{\eta}}$, but if we have instead
\begin{equation} \label{eq:aux}
f \in D_r(\fdag)\quad \mbox{and} \quad f-\fdag \in \X_{\frac{a}{\eta}} \quad \mbox{such that} \quad \|f-\fdag\|_{\frac{a}{\eta}} \le \tau \quad \mbox{for}\quad \tau<\left(\frac{\bar{K}\,\check{K}}{2}\right)^{-1/\eta}.
\end{equation}
Unfortunately, it is difficult to exploit the conditional stability estimate \eqref{eq:stability2} in that case for the stable approximate solution of equation \eqref{eq:opeq}, because approximate solutions
$f$ to $\fdag$ have to satisfy the condition \eqref{eq:aux}. This, however, cannot be expected if $f$ is a regularized solution $\falhat$ from \eqref{eq:tik}, independent of the choices of $s$ and $\alpha$.
\end{rem}

Let us conclude this example with the consideration of the special case $\eta=1$ in \eqref{eq:generalnonlinearity2}, which requires Lipschitz continuity of the Fr\'echet derivatives in $D_r(\fdag)$.
Then we assume $\fdag \in \X_a \cap D(F)$, and the admissible set in the conditional stability estimate \eqref{eq:stability2} attains the form $Q=D_r(\fdag) \cap D_\tau^a(\fdag)$ with a radius $\tau<\frac{2}{\bar{K}\,\check{K}}$ of the second intersected ball. As an illustration of this case we briefly recall the \textit{autoconvolution operator}
\begin{equation} \label{eq:auto}
[F(f)](s)= \int \limits_0^s f(s-t)\,f(t)\, dt \quad (0 \le t \le 1)
\end{equation}
mapping in $\X=\Y=L^2(0,1)$ with domain $D(F)=\X$, which was comprehensively analyzed in the literature with applications in statistics, spectroscopy and laser optics (see, e.g., \cite{GoHo94} and \cite{BuerHof15,Flemmingbuch18}). For noncompact nonlinear operator $F$ from \eqref{eq:auto} we have the compact Fr\'echet derivative
$$[F^\prime (f) h](s)=2\, \int_0^s f(s-t)\,h(t)d t  \quad (0 \le s \le 1,\;\; h \in \X). $$
This gives the relations
$$\|F(f)-F(\fdag)-F^\prime(\fdag)(f-\fdag)\|_{\Y}= \|F(f-\fdag)\|_{\Y} \le \|f-\fdag\|^2  \qquad \mbox{for all} \quad f \in \X,$$
which coincides with \eqref{eq:Hoeldernonlinearity} for $\check{K}=2$ and $\eta=1,$ but here for arbitrary large radii $r>0$.

Now we restrict our focus on the autoconvolution problem further to the specific solution
\begin{equation} \label{eq:1}
\fdag(t) \equiv 1 \quad (0 \le t \le 1),
\end{equation}
where
$$\|F^\prime(\fdag)\,h\|_{\Y} =2\,\|J\,h\|_{\Y}= 2\,\|(J^*J)^{1/2}\,h\|_{\X}  \qquad \mbox{for all} \quad h \in \X,$$
with the simple integration operator $J$ from \eqref{eq:J} such that $L=(J^*J)^{-1/2}$ defines the Hilbert scale. Hence, we have
$$\|f-\fdag\|_{-1}=\frac{1}{2}\,\|F^\prime(\fdag)(f-\fdag)\|_{\Y}\qquad \mbox{for all} \quad f \in \X.$$
Unfortunately, for the solution $\fdag$ from \eqref{eq:1} with $\fdag(1) \not=0$ we have  $\fdag \in \X_\nu$ only for $\nu<1/2$ and hence $\fdag \notin \X_1$ (cf. e.g. \cite[Lem. 8]{gy99}). Then Proposition~\ref{pro:Q2} cannot be applied immediately in case of this solution $\fdag$, but Remark~\ref{rem:nonsmooth} is applicable and  for $a=1,\;\bar K=\frac{1}{2},\; \check{K}=2,$ $\eta=1,\;\tau<2$ and arbitrarily large $r>0$ we get for the autoconvolution operator $F$ from \eqref{eq:auto} and $\fdag$ from \eqref{eq:1} the conditional stability estimate
\begin{equation} \label{eq:autoCSE}
\|f-\fdag\|_{-1} \le \frac{1}{2-\tau}\,\|F(f)-F(\fdag)\|_{\Y}
\end{equation}
if $f \in D(F)$ and  $f-\fdag \in \X_1$ such that $\|f-\fdag\|_1 \le \tau.$
\end{ex}

\subsection*{Variant (B): based on global inequalities of the forward operator $F$}

\begin{ex}[$Q=D^\theta_\rho(0),\;a=0,\;\theta>0$, global conditional estimates] \label{ex:Q3}
The situation of this example is essentially different from those of Examples~\ref{ex:Q1} and \ref{ex:Q2}, because the focus is now on more global conditional stability estimates
\begin{equation}\label{eq:globalstability}
\|f - \tilde f\| \leq R \, \varphi\left(\|F(f)-F(\tilde f)\|_{\Y}\right) \qquad \mbox{for all} \quad f,\tilde f \in D_\rho^\theta(0),
\end{equation}
where $\theta$ is a positive number and the multiplier $R$ may depend on the radius $\rho>0$. For $\tilde f=\fdag$, the estimate \eqref{eq:globalstability} is a special case of \eqref{eq:stability}
with $a=0$. The corresponding set $Q$ collects elements $f \in D(F)$ with the property $\|f\|_\theta \le \rho$. Hence, all the consequences of \eqref{eq:stability} concerning approximate (regularized) solutions are valid under the condition \eqref{eq:globalstability}, too. However, under \eqref{eq:globalstability} such consequences are \textit{uniformly} valid for all $\fdag \in Q$ and
moreover no derivatives of the forward operator $F$ are required.

Based on the seminal paper \cite{ChengYamamoto00} the stable approximate solution of inverse problems under a conditional stability estimate \eqref{eq:globalstability} was studied by numerous authors in the past years.
Estimates of the form \eqref{eq:globalstability} can be verified in large numbers for parameter identification problems in differential equations by powerful tools of PDE theory like Carleman estimates.
With respect to concrete applications we refer to \cite{ChengHofmannLu14,ChengYamamoto00} and further literature mentioned
 therein. For a glimpse of such examples, we briefly recall here in the following a parameter identification problem, which was
comprehensively outlined in \cite{HofmannYamamoto10} (see also \cite[Sect.~5.2]{EggerHof18}).

We consider for $\X=\Y=L^2(0,T)$ the identification of the parameter function $f \in \X$ in the reaction-diffusion problem
$$\begin{array}{ccl}
\partial_t u(\xi,t)- \Delta u(\xi,t)+f(t)\,u(\xi,t)=0 & \quad & \mbox{for}\quad\xi\in \Omega,\;0<t\le T,\\
 \partial_n u(\xi,t)=0 & & \mbox{for}\quad\xi \in \partial \Omega,\;0 <t  \le T,\\
 u(\xi,0)=u_0(\xi) & &  \mbox{for}\quad\xi \in \Omega,
\end{array}$$
from integral data $$g(t)=\int \limits_{\xi \in \Omega} u(\xi,t) d\xi\qquad (0 \le t \le T)$$
of the state variable $u$. The mapping $f \in L^2(0,T) \mapsto \int \limits_{\xi \in \Omega} u(\xi,\cdot) d\xi \in L^2(0,T) $
defines the forward operator  $F:D(F) \subset \X \to \Y$ with domain $D(F)=\{f \in L^2(0,T)~\big|~f \ge 0 \;\;\mbox{a.e.}\}$.
Then for $1/2<\theta<1$ one can show the existence of a constant $R=R(\rho)$ such that
$$\|f-\tilde f\|_{L^2(0,T)} \le R\,\|F(f)-F(\tilde f)\|^{\frac{\theta}{\theta+1}}_{L^2(0,T)} $$
whenever $f,\tilde f \in D(F)$ and $\|f\|_{H^\theta(0,T)}\le \rho,\;\|\tilde f\|_{H^\theta(0,T)}\le \rho$.
This is a H\"older-type conditional stability estimate of the form \eqref{eq:globalstability} with the strictly concave index function $\varphi(t)=t^\frac{\theta}{\theta+1}$ if the Hilbert scale is generated in such a way that
$X_\nu=H^\nu(0,T)$ for $0 \le \nu \le 1$.
\end{ex}

\begin{ex}[Global conditional estimates, relation to variational source conditions]\label{ex:stab}
As a final example, which is also related to variational source conditions, we consider the problem of reconstructing the refractive index $n = 1-\fdag$ from far field data $u^\infty$ in the acoustic scattering problem
\begin{subequations}
\begin{align}
u(x) & = \exp\left(\textup{i} \kappa x \cdot d\right)+ u^s \left(x\right), \label{eq:total_field}\\
\Delta u + \kappa^2 n u & = 0 && \text{in }\R^3, \\
\frac{\partial u^s}{\partial r} - \textup{i} \kappa u^s & = \mathcal O \left(\frac{1}{r^2}\right) &&\text{as } r= \left|x\right| \to \infty, \label{eq:sommerfeld} \\
u\left(x\right)&= \exp\left(\textup{i} \kappa x \cdot d\right) + \frac{\exp\left(\textup{i}\kappa r\right)}{r} \left(u^\infty \left(\hat x\right) + \mathcal O \left(\frac{1}{r^2}\right)\right) && \text{as } r= \left|x\right| \to \infty,
\end{align}
\end{subequations}
where the so-called \textit{Sommerfeld radiation condition} \eqref{eq:sommerfeld} is assumed to hold \textit{uniformly} for all directions $\hat x = x/r \in \mathbb S^2 = \left\{x \in \R^3 \mid \left|x\right| = 1\right\}$.

In practical applications, either one or several incident directions $d \in \mathbb S^2$ can be measured. Here we consider all directions $d$ as available and define $F \left(\fdag\right):= u^\infty$. This forward operator can be seen as a mapping from $L^\infty\left(\R^3\right)$ to $L^2 \left(\mathbb S^2 \times \mathbb S^2\right)$, and its natural domain of definition is
\[
D\left(F\right) := \left\{f \in L^\infty \left(\R^3\right)\mid \Im \left(f\right)\leq 0, \Re \left(f\right) \leq 1, \supp \left(f\right) \subset \left\{x \in \R^3 \mid \left|x\right| \leq \pi\right\}\right\},
\]
as for all $f \in D\left(F\right)$ the problem \eqref{eq:total_field}--\eqref{eq:sommerfeld} admits a unique solution.

For this problem, it has been shown in \cite{HohWei15} (cf. Theorem 2.4 and Corollary 2.5 ibidem) that a variational source condition and a conditional stability estimate hold true. More precisely, if $3/2 < m < s$ such that $s \neq 2 m + 3/2$ and $\fdag \in D \left(F\right) \cap H^s \left(\R^3\right)$ with the Fourier-based Sobolev space $H^s \left(\R^3\right)$, then the variational source condition \eqref{eq:vsc} with $-a = m$ and the function
\[
\varphi\left(t\right) = A \left(\ln\left(3 + t^{-1} \right)\right)^{-2\mu \theta}, \qquad \mu = \min\left\{1, \frac{s-m}{m+3/2}\right\}
\]
holds true for any $0 < \theta < 1$. This also implies the conditional stability estimate \eqref{eq:stability} with $-a = m$ and $\varphi$ as in the above formula. Note that the case $-a = m$ corresponds to $a < 0$ and is hence not covered by our analysis.
\end{ex}

\section*{Acknowledgments}
BH is supported by German Research Foundation (DFG) via grant HO 1454/12-1, and FW has also been supported by the DFG through CRC 755, subproject A07.

\end{document}